\documentclass[reqno,12pt]{amsart}

\usepackage{amscd,amssymb}

\begin{document}
\title{A construction of projective  bases  for irreducible representations of multiplicative groups of division algebras over local fields}

\maketitle

\newcommand{\set}[1]{\left\{#1\right\}}

\newcommand{\un}{{\underline}}

\newcommand{\mR}{{\mathbb R}}

\newcommand{\mC}{{\mathbb C}}

\newcommand{\mZ}{{\mathbb Z}}

\newcommand{\mP}{\mathbb P}

\newcommand{\mG}{\mathbb G}

\newcommand{\mA}{\mathbb A}

\newcommand{\mF}{\mathbb F}

\newcommand{\mQ}{\mathbb Q}

\newcommand{\ho}{\hookrightarrow}

\newcommand{\Gg}{\gamma}

\newcommand{\GG}{\Gamma}

\newcommand{\bO}{\Omega}

\newcommand{\bo}{\omega}

\newcommand{\bl}{\lambda}

\newcommand{\bL}{\Lambda}

\newcommand{\bs}{\sigma}

\newcommand{\bS}{\Sigma}

\newcommand{\ep}{\epsilon}

\newcommand{\D}{\Delta}

\newcommand{\mcA}{\mathcal A}

\newcommand{\mcB}{\mathcal B}

\newcommand{\mcC}{\mathcal C}

\newcommand{\mcD}{\mathcal D}

\newcommand{\mcE}{\mathcal E}

\newcommand{\mcF}{\mathcal F}

\newcommand{\mcG}{\mathcal G}

\newcommand{\mcH}{\mathcal H}

\newcommand{\mcI}{\mathcal I}

\newcommand{\mcJ}{\mathcal J}

\newcommand{\mcK}{\mathcal K}

\newcommand{\mcL}{\mathcal L}

\newcommand{\mcM}{\mathcal M}

\newcommand{\mcN}{\mathcal N}

\newcommand{\mcO}{\mathcal O}

\newcommand{\mcP}{\mathcal P}

\newcommand{\mcQ}{\mathcal Q}

\newcommand{\mcR}{\mathcal R}

\newcommand{\mcS}{\mathcal S}

\newcommand{\mcT}{\mathcal T}

\newcommand{\mcU}{\mathcal U}

\newcommand{\mcV}{\mathcal V}

\newcommand{\mcW}{\mathcal W}

\newcommand{\mcX}{\mathcal X}

\newcommand{\mcY}{\mathcal Y}

\newcommand{\mcZ}{\mathcal Z}

\newcommand{\fa}{\mathfrak a}

\newcommand{\fb}{\mathfrak b}

\newcommand{\fc}{\mathfrak c}

\newcommand{\fd}{\mathfrak d}

\newcommand{\fe}{\mathfrak e}

\newcommand{\ff}{\mathfrak f}

\newcommand{\fg}{\mathfrak g}

\newcommand{\fh}{\mathfrak h}

\newcommand{\fii}{\mathfrak i}

\newcommand{\fj}{\mathfrak j}

\newcommand{\fk}{\mathfrak k}

\newcommand{\fl}{\mathfrak l}

\newcommand{\fm}{\mathfrak m}

\newcommand{\fn}{\mathfrak n}

\newcommand{\fo}{\mathfrak o}

\newcommand{\fp}{\mathfrak p}

\newcommand{\fq}{\mathfrak q}

\newcommand{\fr}{\mathfrak r}

\newcommand{\fs}{\mathfrak s}

\newcommand{\ft}{\mathfrak t}

\newcommand{\fu}{\mathfrak u}

\newcommand{\fv}{\mathfrak v}

\newcommand{\fw}{\mathfrak w}

\newcommand{\fx}{\mathfrak x}

\newcommand{\fy}{\mathfrak y}

\newcommand{\fz}{\mathfrak z}
\newcommand{\kk}{\kappa}

\newcommand{\ti}{\tilde}
\newcommand{\und}{\underline}

\newcommand{\va}{\vartriangleleft}
\newcommand{\bas}{\backslash}

\newtheorem{theorem}{Theorem}[section]

\newtheorem{proposition}[theorem]{Proposition}

\newtheorem{conjecture}[theorem]{Conjecture}

\newtheorem{corollary}[theorem]{Corollary}

\newtheorem{lemma}[theorem]{Lemma}

\newtheorem{claim}[theorem]{Claim}

\newtheorem{maintheorem}[theorem]{Main Theorem}

\newtheorem{question}[theorem]{Question}

\theoremstyle{definition}

\newtheorem{remark}[theorem]{Remark}

\newtheorem{example}[theorem]{Example}

\newtheorem{definition}[theorem]{Definition}

\newtheorem{algorithm}[theorem]{Algorithm}

\newtheorem{problem}[theorem]{Problem}

\newtheorem{pf}[theorem]{Proof}

\author{ David Kazhdan}

\begin{abstract}
  Let $F_0$ be a local non-archimedian field of positive characteristic, $D_0$ be a skew-field with  center $F_0$ and $ G_0:=D_0^{\star}$ be the multiplicative group of
$D_0$.
The goal of this paper is to provide a canonical decomposition of 
any complex irreducible representation $V$ of $G_0$ in a direct sum of one-dimensional subspaces.

\end{abstract}

Let $k=\mF _q$ be a finite field, $F:=k(t)$ the field of rational functions on  the projective line $\underline \mP ^1$ over $k$. 
Let $S$ be the set of points of $\underline \mP ^1$.  For any point $s\in S$ we denote by $F_s$ the completion of $F$ at $s$, by $\nu _s:F_s\to \mZ \cup +\infty$ the valuation map 
 and by $\mcO _s \subset F_s$ the subring of integers. We denote by $\mA$ the ring of adeles for $F$.

Let $D$ be a  skew-field  with  center $F$
unramified outside $\{ 0,{\infty}\},D_0:=D\otimes _FF_0$ and $D_{\infty}:=D\otimes _FF_{\infty}$. We have $dim_F(D)=n^2$.

We denote  by $\und G$ the mutiplicative group of $D$ considered as an algebraic $F$-group, write $G_s:=\und G(F_s)$ and denote by $\mC _cG_0)\subset \mC (G_0)$ the algebra of locally constant compactly supported functions on $G_0$.

For any point $s\in S, s\neq 0,\infty$ we identify the group the group 
$\und G(F_s)$ with $Gl_n(F_s)$ and define $K_s:=Gl_n(\mcO _s)$.

Let $N_\infty :G_\infty \to F_\infty ^\star$ be the reduced norm. We define $K \infty =\{ g\in D_\infty | \nu _\infty (N_\infty (g))\geq 0\} $
 and $K^1_\infty =\{ g\in D_\infty | \nu _\infty (N_\infty (g-1))>0\}$.
Then $K_\infty \subset G_\infty$ is an open compact subgroup and $K_\infty /K^1_\infty =\mF _{q^n}^\star$. We define 
$K^1:=\prod _{s\in S -\{0,\infty \}}K_s\times K^1_\infty$.

The multiplication defines a map 
$$\kk :G_0\times K^1\times \und G(F)\to \und G(\mA)$$
This paper is based on the following result
\begin{proposition}The map $\kk$ is a bijection.
\end{proposition}
\begin{proof}
The surjectivity follows from Lemma 7.4 in \cite{HK}. To show the injectivity
it is sufficient to check the equality

$$(D_0^{*}\times K^1)\cap G_E=\{ e\}$$

 which is obvious.

\end{proof}
Let $R$ be the space of $\mC$-valued locally constant functions on 
$\und G(\mA )/\und G(F) K^1$,
 $\mcH_s, s\neq 0,\infty $ be the spheical Hecke algebra at $s, \mcH : =\prod _{s\in S -\{0,\infty \}} \mcH _s$.  

We have a natural action $a\to \hat a$ of the commutative algebra $\mcA :=\mcH \otimes \mC [K^0_\infty /K^1_\infty ]$ on $R$.

\begin{corollary} 
\begin{enumerate}
\item The natural action of the group $G_0$ on the space $X:=\und G(\mA)/ K^1\times \und G(F)$ is simply transitive. So we can identify $X$ with $G_0$.
\item The restriction to $G_0$ defines a $G_0$-equivariant isomorphism 
$u:R\to \mC (G_0)$.
\item For  any irreducible representation $V$ of $G_0$  the restriction to $G_0$ defines an isomorphism $u_V: Hom _{G_0}(V^\vee ,L) \to V $
where $V^\vee$ is the representation dual to $V$.
\item There exists a map $\alpha :\mcA \to \mC _c(G_0)$ such that 
$$f\star \alpha (a))=\hat a(f), a\in \mcA ,f\in \mC _c(G_0)$$
\end{enumerate}
\end{corollary}

Let $\Xi$ be the set of homomorphisms 
$\chi : \mcH \otimes \mC [\mF _{q^n}^\star] \to \mC$. For any
 $\chi \in \Xi$  we define $V_\chi :=\{ v\in V|av=\chi (a)v\}$ for all 
$a\in \mcH \otimes \mC [\mF _{q^n}^\star] $.
 Let $\Xi _V=\{ \chi \in \Xi |V_\chi \neq \{ 0\}\}$.
\begin{theorem}
\begin{enumerate}
\item $dim (V_\chi )=1$ for all $\chi \in \Xi _V$.
\item $V=\oplus _{\chi \in \Xi _V} V_\chi $.
\end{enumerate}
\end{theorem}
\begin{proof}

As follows from \cite{DKV} and \cite{PS} we have direct sum decomposition
 $$V=\oplus _{\chi \in \Xi _V}V_\chi$$
where the subspaces $V_\chi \subset V$ are 
$\mcH \times G_\infty$-invariant and the representation $\ti \rho _\chi$ of $\mcH \times G_\infty$ on $V_\chi$ is irreducible. Since $\mcH $ is commutative this implies the irreduciblity of the restriction $\rho _\chi$ of 
$\ti \rho _\chi$ to $G_\infty $. By definition we can consider $\rho _\chi$ as a representation of the quotient group  $G_\infty /K^1_\infty =\mZ \ltimes \mF^\star  _{q^n}$ where $1\in \mZ$ acts by the Frobenious automorphism on $ \mF^\star  _{q^n}$. It is easy to see that the restriction of any irreducible representation of the group $\mZ \ltimes \mF^\star  _{q^n} $ on 
$ \mF^\star  _{q^n}$ is the direct sum of distinct one-dimensional representations.
\end{proof}

\begin{question} Is the subalgebra $\alpha (\mcA)\subset \mC _c(G_0)$ invariant under the natural action of the group of automorphisms of $F$?
\end{question}

\begin{remark} The paper \cite{K} was influnced by \cite{Katz} and is concerned with the understanding of the local Langlands conjecture. This short paper is a streamlined version of \cite{K}.

Acknowledgments. The project has received funding from ERC under grant agreement 669655. 
\end{remark}

\end{document}